\documentclass[11pt,a4paper]{amsart}
\usepackage[T1]{fontenc}
\usepackage{lmodern,amssymb,mathtools,booktabs}
\usepackage[expansion=false]{microtype}
\usepackage[margin=29mm]{geometry}
\usepackage[colorlinks=true,linkcolor=blue,citecolor=blue,urlcolor=blue]{hyperref}
\hypersetup{pdftitle={A probabilistic construction of maximum h-scattered subspaces},pdfauthor={Alessandro Giannoni and Alessandro Neri}}
\newtheorem{theorem}{Theorem}[section]
\newtheorem{proposition}[theorem]{Proposition}
\newtheorem{lemma}[theorem]{Lemma}
\newtheorem{corollary}[theorem]{Corollary}
\theoremstyle{definition}
\newtheorem{remark}[theorem]{Remark}
\newtheorem{example}[theorem]{Example}
\numberwithin{equation}{section}
\newcommand{\F}{\mathbb F}
\newcommand{\rk}{\operatorname{rk}}
\newcommand{\GL}{\operatorname{GL}}
\newcommand{\Mat}{\operatorname{Mat}}
\newcommand{\PG}{\operatorname{PG}}
\newcommand{\codim}{\operatorname{codim}}
\newcommand{\tr}{\operatorname{tr}}
\newcommand{\Tr}{\operatorname{Tr}}
\newcommand{\Gal}{\operatorname{Gal}}
\newcommand{\lcm}{\operatorname{lcm}}
\newcommand{\Frob}{\operatorname{Frob}}
\newcommand{\Stab}{\operatorname{Stab}}

\newcommand{\Prob}{\mathbb P}
\newcommand{\gauss}[2]{\genfrac{[}{]}{0pt}{}{#1}{#2}_{q}}
\newcommand{\gaussf}[3]{\genfrac{[}{]}{0pt}{}{#1}{#2}_{#3}}

\title{A probabilistic construction of maximum $h$-scattered subspaces}

\author{Alessandro Giannoni}
\address{Dipartimento di Matematica e Applicazioni ``Renato Caccioppoli'', Universit\`a degli Studi di Napoli Federico II, Via Cintia, Monte S.~Angelo, 80126 Napoli, Italy}
\email{[alessandro.giannoni@unina.it]}
\author{Alessandro Neri}
\address{Dipartimento di Matematica e Applicazioni ``Renato Caccioppoli'', Universit\`a degli Studi di Napoli Federico II, Via Cintia, Monte S.~Angelo, 80126 Napoli, Italy}
\email{alessandro.neri@unina.it}
\date{}

\begin{document}

\maketitle

{
\begin{abstract}
Let $1\le h\le n-2$ and $g=\gcd(n,h+1)$. We prove that, for every prime power $q$, a maximum $h$-scattered $\mathbb{F}_q$-subspace of $V(r,q^n)$ exists whenever $(h+1)\mid rn$ and
$$
r\ge\frac{h+1}{g}\bigl(h(n-h-2)+1\bigr).
$$
In particular, for fixed $n$ and minimum distance $\delta\in\{2,\ldots,n-1\}$, $\mathbb{F}_{q^n}$-linear MRD codes of length $N$ exist for every $q$ and every sufficiently large $N$ such that $n$ divides $(n-\delta+1)N$, with a bound on $N$ depending only on $n$ and $\delta$. We take the $\mathbb{F}_{q^s}$-span of the columns of a random matrix $A$ over the algebra $R=\mathbb{F}_{q^n}\otimes_{\mathbb{F}_q}\mathbb{F}_{q^s}$, for a suitable $s$. Such a subspace is maximum $h$-scattered if and only if the kernel of $A$ contains no nonzero vector whose expansion with respect to an $\mathbb{F}_q$-basis of $\mathbb{F}_{q^s}$ has rank weight at most $h+1$. When $\gcd(n,s)=1$, $R$ is a field, and a first-moment argument, counting these vectors up to multiplication by $\mathbb{F}_{q^n}^*\mathbb{F}_{q^s}^*$, shows that with positive probability the kernel contains none of them if $s>h(n-h-2)$. When $d=\gcd(n,s)>1$, $R$ is a product of $d$ fields, and vectors vanishing in some factors lie in the kernel with larger probability. We control them by proving that, for an $(h+1)$-dimensional $\mathbb{F}_q$-subspace $D$ of $\mathbb{F}_{q^n}$, the elements of $D\otimes_{\mathbb{F}_q}\mathbb{F}_{q^s}$ vanishing outside $k$ given factors, $0<k<d$, form an $\mathbb{F}_{q^s}$-space of dimension at most $(h+1)k/d$, with equality only if $D$ is linear over a subfield of $\mathbb{F}_{q^n}$ larger than $\mathbb{F}_q$. Taking $s=gr/(h+1)$ gives the bound on $r$. Finally, if $(n-h-1)\mid(h+1)$, then for every $r$ satisfying the above inequality there is an $h$-scattered $\mathbb{F}_q$-subspace of $V(r,q^n)$ of dimension $\lfloor rn/(h+1)\rfloor$, without a divisibility condition on $rn$.
\end{abstract}}

\section{Introduction}

Let $V=V(r,q^n)$ be an $r$-dimensional vector space over $\F_{q^n}$. An $\F_q$-subspace $U$ of $V$ is called \emph{$h$-scattered} if $\langle U\rangle_{\F_{q^n}}=V$ and
\[
\dim_{\F_q}(U\cap W)\le h
\]
for every $h$-dimensional $\F_{q^n}$-subspace $W$ of $V$. Throughout this paper, we assume
\[
1\le h\le n-2,\qquad r\ge h+1.
\]
In this range, it follows from a theorem of Csajb\'ok, Marino, Polverino and Zullo \cite[Theorem~2.3]{CMPZ} that
\[
\dim_{\F_q}U\le\frac{rn}{h+1}.
\]
We call $U$ \emph{maximum} when equality holds, and we call $r$ \emph{admissible} when $(h+1)\mid rn$. When $(h+1)\nmid rn$, we ask instead whether $U$ can attain the integer upper bound $\lfloor rn/(h+1)\rfloor$.

Maximum $h$-scattered subspaces correspond to linear MRD codes \cite{ZZ,MNT,NZ}. If an $\F_q$-basis of $U$ is used as the columns of a matrix over $\F_{q^n}$, its row space is an $\F_{q^n}$-linear MRD code with parameters
\[
[N,r,n-h]_{q^n/q},\qquad N=\frac{rn}{h+1}.
\]
Conversely, by Proposition~\ref{prop:criterion} and the duality of MRD codes, every $\F_{q^n}$-linear MRD code with these parameters arises in this way. Thus maximum $h$-scattered subspaces of $V(r,q^n)$ exist if and only if $\F_{q^n}$-linear MRD codes with parameters $[rn/(h+1),r,n-h]_{q^n/q}$ exist.

Related existence and enumeration problems have also been studied through rank-metric lattices. Cotardo and Ravagnani \cite[Theorem~23 and Section~5]{CR} introduced these lattices and related their Whitney numbers to code counts, and their critical exponents to the largest dimensions attainable for a prescribed minimum rank distance. Cotardo, Ravagnani and Zullo \cite{CRZ} used hyperovals and linear sets to obtain further Whitney number formulas and density estimates. In this language, the condition on $\ker G$ in Proposition~\ref{prop:criterion} below says that $\ker G$ contains no atom of the rank-metric lattice of $\F_{q^n}^N$ generated by the vectors of rank at most $h+1$.

When $(h+1)\mid r$, direct sums of the standard construction in $V(h+1,q^n)$ give maximum $h$-scattered subspaces \cite[Theorem~2.6]{CMPZ}; further families in these dimensions appear in \cite{BGM}. If $\gcd(n,h+1)=1$, then $(h+1)\mid rn$ forces $(h+1)\mid r$, so the existence problem is settled in this case. The case $h=1$ is also known for all admissible parameters \cite{BL,BBL,BGMP,CMPZ17}, and Delsarte duality gives the corresponding result for $h=n-3$ \cite[Corollary~3.5]{CMPZ}. The case $(n,h)=(6,2)$ is settled for every $q$ and every $r\ge3$ in \cite{BGGMN}; earlier, \cite{BGGM} gave a construction for $r=4$ and $q$ an odd power of $2$, which settled the case $r\ge3$, $r\ne5$, for these $q$.

For $r\ge5$ coprime to $6$, the construction in \cite{BGGMN} uses a two-dimensional $\F_{q^r}$-subspace of $\F_{q^{6r}}$. In this paper we use such an auxiliary field for general parameters. Let $s$ be coprime to $n$, and let $\ell,t$ satisfy $n\ell=(h+1)t$. Then $\F_{q^{ns}}^\ell$ is an $s\ell$-dimensional $\F_{q^n}$-space, and an $\F_q$-basis of $\F_{q^s}$ is an $\F_{q^n}$-basis of $\F_{q^{ns}}$. We take the $\F_{q^s}$-span of the columns of a random $\ell\times t$ matrix over $\F_{q^{ns}}$; it is an $\F_q$-subspace of $V(s\ell,q^n)$. It fails to be maximum $h$-scattered if and only if the matrix over $\F_{q^n}$ obtained by expanding the columns in an $\F_q$-basis of $\F_{q^s}$ has a nonzero kernel vector of rank at most $h+1$.

The main counting step groups these relations under multiplication by $\F_{q^n}^*\F_{q^s}^*$. This reduces the first-moment bound to
\[
\gauss{n}{h+1}\frac{q-1}{(q^n-1)(q^s-1)},
\]
which is at most one precisely when $s>h(n-h-2)$. To use degrees not coprime to $n$, we replace the compositum by
\[
\F_{q^n}\otimes_{\F_q}\F_{q^s}
\cong\prod_{i=1}^{\gcd(n,s)}\F_{q^{ns/\gcd(n,s)}}.
\]
A vector supported on only some of these factors lies in the kernel of a random matrix with larger probability. We bound the number of such vectors by an inequality between dimensions and supports (Theorem~\ref{thm:etale}), and we determine its equality case (Theorem~\ref{thm:etale-equality}). This leads to the following result.

\begin{theorem}\label{thm:main}
Let $1\le h\le n-2$ and $g=\gcd(n,h+1)$. For every prime power $q$, a maximum $h$-scattered $\F_q$-subspace of $V(r,q^n)$ exists whenever
\[
(h+1)\mid rn,\qquad
r\ge\frac{h+1}{g}\bigl(h(n-h-2)+1\bigr).
\]
\end{theorem}

The new argument is needed only when $g>1$ and $2\le h\le n-4$; in the remaining cases Theorem~\ref{thm:main} follows from the constructions recalled above. For example, Theorem~\ref{thm:main} gives maximum $3$-scattered subspaces of $V(r,q^8)$ for every $q$ and every $r\ge10$, in addition to the standard dimensions $r=4,8$. Using only the coprime degrees gives the threshold $r\ge19$.

The coding-theoretic formulation appears in Corollary~\ref{cor:mrd}. It proves an asymptotic form of the long-length case of the linear MRD conjecture (see, e.g., \cite[Conjecture~4.2(2)]{NZ}): for fixed extension degree and minimum distance, all sufficiently large admissible lengths occur, with a threshold independent of $q$.

Adding coordinates over $\F_q$ to the subspaces of Theorem~\ref{thm:main}, we obtain the following result, in which $(h+1)\mid rn$ is not required.

\begin{theorem}\label{thm:mainrounded}
Suppose $(n-h-1)\mid(h+1)$. For every prime power $q$ and every
\[
r\ge\frac{h+1}{n-h-1}\bigl(h(n-h-2)+1\bigr),
\]
there exists an $h$-scattered $\F_q$-subspace of $V(r,q^n)$ of dimension $\lfloor rn/(h+1)\rfloor$.
\end{theorem}

More generally, this procedure gives explicit residue classes of $r$ in which the integer upper bound is attained (Corollary~\ref{cor:residue-classes}). For $h=n-2$, and for $h=n-3$ with $n$ even, the known maximum constructions together with this procedure attain the integer upper bound for every $r\ge h+1$. For $h=n-3$ with $n\ge7$ odd, the auxiliary field gives further residue classes (Corollary~\ref{cor:oddm2}).

Section~\ref{sec:relations} records the relation criterion and the direct-sum construction. Section~\ref{sec:first} gives the coprime-field count. Section~\ref{sec:etale} proves the support inequality, its equality case and Theorem~\ref{thm:main}. Section~\ref{sec:rounded} treats the integer upper bound and proves Theorem~\ref{thm:mainrounded}.

\section{Low-rank relations and scattered subspaces}\label{sec:relations}

For $x=(x_1,\ldots,x_N)\in\F_{q^n}^N$, its rank weight is
\[
\rk_q(x)=\dim_{\F_q}\langle x_1,\ldots,x_N\rangle_{\F_q}.
\]
An $\F_{q^n}$-linear rank-metric code of length $N$ and dimension $k$ is a $k$-dimensional $\F_{q^n}$-subspace of $\F_{q^n}^N$. Its minimum distance is the least rank weight of a nonzero codeword. We use the notation $[N,k,\delta]_{q^n/q}$.

After choosing an $\F_q$-basis of $\F_{q^n}$, such a code becomes an $\F_q$-linear space of $n\times N$ matrices. For $N\ge n$, the Singleton bound is
\[
nk\le N(n-\delta+1).
\]
Indeed, deleting $\delta-1$ rows is injective on the matrix code. Codes attaining this bound are called MRD. By \cite{Delsarte}, the dual of an MRD code with respect to the pairing $(X,Y)\mapsto\tr(XY^{\mathsf T})$ is MRD. If the coordinates of the two codes are expanded in an $\F_q$-basis of $\F_{q^n}$ and in its trace-dual basis, respectively, this pairing becomes $(x,y)\mapsto\Tr_{q^n/q}(\sum_i x_iy_i)$. If $C$ is $\F_{q^n}$-linear, then $\Tr_{q^n/q}(\sum_i x_iy_i)=0$ for all $x\in C$ if and only if $\sum_i x_iy_i=0$ for all $x\in C$, since $C$ is closed under $\F_{q^n}$-scalars and the trace form is nondegenerate. Hence the ordinary dual of an $\F_{q^n}$-linear MRD code is MRD.

\begin{lemma}\label{lem:independence}
Let $U\le_{\F_q}\F_{q^n}^r$ have $\F_{q^n}$-span $\F_{q^n}^r$, and let $1\le h<r$. Then $U$ is $h$-scattered if and only if every $h+1$ vectors of $U$ that are independent over $\F_q$ are independent over $\F_{q^n}$.
\end{lemma}
\begin{proof}
If $h+1$ independent vectors over $\F_q$ are dependent over $\F_{q^n}$, their $\F_{q^n}$-span is contained in an $h$-dimensional $\F_{q^n}$-space. Conversely, an intersection of $\F_q$-dimension at least $h+1$ supplies such a set of vectors.
\end{proof}

\begin{lemma}\label{lem:generalcriterion}
Let $r\ge h+1$, let $N\ge1$, and let $G\in\Mat_{r\times N}(\F_{q^n})$. The $\F_q$-span $U$ of its columns is $h$-scattered of dimension $N$ if and only if $G$ has row rank $r$ and its kernel contains no nonzero vector of rank weight at most $h+1$.
\end{lemma}
\begin{proof}
Suppose that the two matrix conditions hold. The absence of a rank-one vector in the kernel makes the columns independent over $\F_q$, and full row rank gives the spanning condition. If $h+1$ vectors $Gz_i$, with $z_i\in\F_q^N$, were independent over $\F_q$ but dependent over $\F_{q^n}$, then the $z_i$ would be independent over both fields. A nontrivial relation would give
\[
0\ne\sum_{i=1}^{h+1}\alpha_i z_i\in\ker G,
\qquad
\rk_q\left(\sum_{i=1}^{h+1}\alpha_i z_i\right)\le h+1,
\]
a contradiction.

Conversely, the assumptions on $U$ give full row rank and column independence over $\F_q$. Write a nonzero kernel vector of rank $b\le h+1$ as $\sum_{i=1}^{b}\alpha_i z_i$, with both $(\alpha_i)$ and $(z_i)$ independent over $\F_q$. The $Gz_i$ are independent over $\F_q$ but span a space of dimension at most $b-1$ over $\F_{q^n}$. Since $U$ spans $\F_{q^n}^r$, one has $N\ge r\ge h+1$. Completing the $Gz_i$ to $h+1$ independent vectors over $\F_q$ therefore contradicts Lemma~\ref{lem:independence}.
\end{proof}

\begin{proposition}\label{prop:criterion}
Let $r\ge h+1$ and $N=\lfloor rn/(h+1)\rfloor$. Let $G\in\Mat_{r\times N}(\F_{q^n})$, and let $U$ be the $\F_q$-span of its columns. The following are equivalent:
\begin{enumerate}
\item $U$ is $h$-scattered in $\F_{q^n}^r$ and $\dim_{\F_q}U=N$;
\item $\ker_{\F_{q^n}}G$ contains no nonzero vector of rank weight at most $h+1$.
\end{enumerate}
In this case, $G$ has rank $r$ and its columns are independent over $\F_q$. If $(h+1)\mid rn$, its row space is an $[N,r,n-h]_{q^n/q}$ MRD code.
\end{proposition}
\begin{proof}
Suppose first that (2) holds. A nonzero vector of $\F_q^N$ has rank weight one, so the columns of $G$ are independent over $\F_q$. Put $k=\rk_{\F_{q^n}}G$. The code $K=\ker_{\F_{q^n}}G$ has $\F_q$-dimension $n(N-k)$ and minimum rank distance at least $h+2$. Since $r\ge h+1$ gives $N=\lfloor rn/(h+1)\rfloor\ge n$, deleting $h+1$ rows gives
\[
n(N-k)\le N(n-h-1).
\]
Put $\epsilon=rn-(h+1)N$, so that $0\le\epsilon\le h<n$. Consequently,
\[
k\ge\frac{(h+1)N}{n}=r-\frac{\epsilon}{n}>r-1,
\]
and hence $k=r$.

Hence $G$ has rank $r$, and (1) follows from Lemma~\ref{lem:generalcriterion}. Conversely, if (1) holds, then $G$ has rank $r$ because $U$ spans $\F_{q^n}^r$, and (2) follows from Lemma~\ref{lem:generalcriterion}.

Finally, if $(h+1)\mid rn$, then $K$ has dimension $N-r=N(n-h-1)/n$, so the Singleton bound forces its minimum distance to be exactly $h+2$. In that case its dual, the row space of $G$, is MRD with minimum distance $n-h$.
\end{proof}

The rank criterion also gives a short proof of the direct-sum property, which is already known from \cite[Theorem~2.4]{CMPZ}.

\begin{corollary}\label{cor:sum}
Let $r_i\ge h+1$ and let $U_i$ be $h$-scattered in $\F_{q^n}^{r_i}$ for $i=1,2$. Then $U_1\oplus U_2$ is $h$-scattered in $\F_{q^n}^{r_1+r_2}$. If both summands are maximum, so is their direct sum.
\end{corollary}
\begin{proof}
Lemma~\ref{lem:generalcriterion} shows that an $\F_q$-basis matrix of each $U_i$ has no nonzero kernel vector of rank at most $h+1$. Take their block diagonal sum. Its columns are independent over $\F_q$, it has full row rank, and its kernel is the direct sum of the two kernels. A nonzero block of a vector has rank weight at most that of the whole vector, so Lemma~\ref{lem:generalcriterion} proves the scattered property. If the two summands are maximum, their dimensions add to the required bound.
\end{proof}

The standard block in dimension $h+1$ is
\[
U_{\mathrm{st}}=\{(x,x^q,\ldots,x^{q^h}):x\in\F_{q^n}\}\le\F_{q^n}^{h+1}.
\]
It has $\F_q$-dimension $n$ and spans $\F_{q^n}^{h+1}$. The Moore determinant shows that any $h+1$ independent vectors over $\F_q$ are independent over $\F_{q^n}$. Hence $U_{\mathrm{st}}$ is maximum $h$-scattered; see \cite[Example~2.5]{CMPZ} and \cite{Gabidulin}.

\section{A random construction with an auxiliary field}\label{sec:first}

Fix a positive integer $s$ coprime to $n$. The fields $\F_{q^n}$ and $\F_{q^s}$ are linearly disjoint over $\F_q$, and their compositum is $\F_{q^{ns}}$. In particular, any $\F_q$-basis
\[
\mathcal B=(b_1,\ldots,b_s)
\]
of $\F_{q^s}$ is an $\F_{q^n}$-basis of $\F_{q^{ns}}$.

Choose a positive integer $\ell$ with $(h+1)\mid n\ell$, and put
\[
t=\frac{n\ell}{h+1},\qquad r=s\ell,\qquad N=st.
\]
We assume $r\ge h+1$, so that $N\ge n$. For $A=(v_1\mid\cdots\mid v_t)\in\Mat_{\ell\times t}(\F_{q^{ns}})$, define
\[
U_A=A\F_{q^s}^t=\F_{q^s} v_1+\cdots+\F_{q^s} v_t\le\F_{q^{ns}}^\ell.
\]
Here $\F_{q^{ns}}^\ell$ is viewed as an $r$-dimensional $\F_{q^n}$-space. No independence is assumed at this stage.

For $y=(y_1,\ldots,y_t)\in\F_{q^{ns}}^t$, write
\[
y_j=\sum_{i=1}^{s}x_{ji}b_i,\qquad x_{ji}\in\F_{q^n},
\]
and set
\[
\Phi_{\mathcal B}(y)=(x_{11},\ldots,x_{1s},\ldots,x_{t1},\ldots,x_{ts}),
\qquad \rk_q(y)=\rk_q(\Phi_{\mathcal B}(y)).
\]
We call $\rk_q(y)$ the \emph{expanded rank} of $y$. It does not depend on the chosen $\F_q$-basis of $\F_{q^s}$, since a change of basis acts on each block of $s$ coordinates by a matrix in $\GL(s,q)$. In general it differs from the rank weight of $y$ as a vector of $\F_{q^{ns}}^t$. From now on, the rank of a vector of $\F_{q^{ns}}^t$ is its expanded rank.

\begin{lemma}\label{lem:invariance}
For $\lambda\in\F_{q^n}^*$, $\mu\in\F_{q^s}^*$ and $y\in\F_{q^{ns}}^t$, one has
\[
\rk_q(\lambda\mu y)=\rk_q(y).
\]
Moreover, $U_A$ is maximum $h$-scattered if and only if $\ker_{\F_{q^{ns}}}A$ contains no nonzero vector of rank at most $h+1$.
\end{lemma}
\begin{proof}
Multiplication by $\lambda$ scales all the coordinates of $\Phi_{\mathcal B}(y)$ by $\lambda$. Multiplication by $\mu$ acts on each block of $s$ coordinates by the same invertible matrix over $\F_q$. Both operations preserve rank.

Fix also an $\F_{q^n}$-linear identification of $\F_{q^{ns}}^\ell$ with $\F_{q^n}^r$. The images of the $N$ vectors $b_i v_j$, ordered as the coordinates of $\Phi_{\mathcal B}$, form the columns of a matrix $G$ over $\F_{q^n}$. Under $\Phi_{\mathcal B}$, the equation $Ay=0$ is exactly $G\Phi_{\mathcal B}(y)^{\mathsf T}=0$. The assertion follows from Proposition~\ref{prop:criterion}.
\end{proof}

Let $\mathcal O$ be the set of points of $\PG(t-1,q^{ns})$ containing a nonzero vector of expanded rank at most $h+1$. Choose $A$ uniformly from $\Mat_{\ell\times t}(\F_{q^{ns}})$. For any fixed point of $\mathcal O$, the probability of lying in $\ker A$ is $q^{-ns\ell}$. The union bound therefore gives
\[
\Prob(U_A\text{ is not maximum }h\text{-scattered})
\le |\mathcal O|q^{-ns\ell}.
\]

\begin{theorem}\label{thm:first}
Let $A$ be chosen uniformly from $\Mat_{\ell\times t}(\F_{q^{ns}})$. Then
\[
\Prob(U_A\text{ is not maximum }h\text{-scattered})\le \beta_{n,h,s}(q)(1-q^{-(h+1)N})
<\beta_{n,h,s}(q),
\]
where
\[
\beta_{n,h,s}(q)=\gauss{n}{h+1}\frac{q-1}{(q^n-1)(q^s-1)}.
\]
Consequently, $\beta_{n,h,s}(q)\le1$ implies the existence of a maximum $h$-scattered subspace of the form $U_A$.
\end{theorem}
\begin{proof}
Let $B_{h+1}$ be the set of nonzero vectors of $\F_{q^n}^N$ of rank at most $h+1$. Every such vector has all its coordinates in some $(h+1)$-dimensional $\F_q$-subspace of $\F_{q^n}$. Therefore
\[
|B_{h+1}|\le\gauss{n}{h+1}(q^{(h+1)N}-1).
\]
The group $H=\F_{q^n}^*\F_{q^s}^*$ has order
\[
|H|=\frac{(q^n-1)(q^s-1)}{q-1},
\]
because $\F_{q^n}\cap\F_{q^s}=\F_q$. Its action by scalar multiplication on nonzero vectors is free. By Lemma~\ref{lem:invariance}, every point of $\mathcal O$ contains an $H$-orbit of vectors of expanded rank at most $h+1$, and $\Phi_{\mathcal B}$ maps the set of all such vectors bijectively onto $B_{h+1}$. Distinct points contain disjoint sets of vectors, whence
\[
|\mathcal O|\le\frac{|B_{h+1}|}{|H|}.
\]
Since $(h+1)N=ns\ell$, the preceding union bound gives the result.
\end{proof}

\begin{lemma}\label{lem:beta}
Let $1\le h\le n-2$. For every positive integer $s$,
\[
\beta_{n,h,s}(q)\le1\quad\Longleftrightarrow\quad s>h(n-h-2).
\]
Moreover, $\beta_{n,h,s}(q)=1$ if and only if $q=2$, $s=1$ and $h=n-2$.
\end{lemma}
\begin{proof}
Set $a=\min\{h+1,n-h-1\}$, so that $n\ge2a$ and $h(n-h-2)=(a-1)(n-a-1)$. Cancelling the factors $q^n-1$ and $q-1$ in the Gaussian coefficient gives
\[
\beta_{n,h,s}(q)
=q^{h(n-h-2)-s}
\frac{\prod_{j=n-a+1}^{n-1}(1-q^{-j})}{(1-q^{-s})\prod_{j=2}^{a}(1-q^{-j})}.
\]
If $a=1$, that is, if $h=n-2$, this is $1/(q^s-1)$ and $h(n-h-2)=0$, proving the claim in this case.

Suppose $a\ge2$. For $1\le i\le a-1$, the numerator factor with index $n-a+i$ is strictly larger than the denominator factor with index $i+1$, since $n-a\ge a\ge2$. Thus $s\le h(n-h-2)$ gives $\beta_{n,h,s}(q)>1$.

It remains to take $s\ge h(n-h-2)+1$. If $a=2$, then $h(n-h-2)=n-3\ge1$, so $s\ge2$, and
\[
\beta_{n,h,s}(q)<\frac{1/2}{(3/4)^2}=\frac{8}{9}<1.
\]
If $a\ge3$, then $h(n-h-2)\ge(a-1)^2\ge4$, so $s\ge5$. Also
\[
\prod_{j=2}^{a}(1-q^{-j})
\ge\frac{3}{4}\prod_{j=3}^{a}(1-2^{-j})
>\frac{3}{4}\left(1-\sum_{j=3}^{\infty}2^{-j}\right)=\frac{9}{16}.
\]
It follows that
\[
\beta_{n,h,s}(q)<\frac{1/2}{(31/32)(9/16)}
=\frac{256}{279}<1.
\]
This proves the statement, including the equality case.
\end{proof}

\begin{corollary}\label{cor:blocks}
If $\gcd(s,n)=1$, $s>h(n-h-2)$, $(h+1)\mid n\ell$, and $s\ell\ge h+1$, then there exists a maximum $h$-scattered subspace of $V(s\ell,q^n)$ that is linear over $\F_{q^s}$.
\end{corollary}
\begin{proof}
By Lemma~\ref{lem:beta}, $\beta_{n,h,s}(q)\le1$, so Theorem~\ref{thm:first} gives a matrix $A$ such that $U_A$ is maximum $h$-scattered. By construction, $U_A$ is an $\F_{q^s}$-subspace of $\F_{q^{ns}}^\ell$.
\end{proof}

\section{Auxiliary algebras without coprimality}\label{sec:etale}

We now drop the assumption $\gcd(n,s)=1$. Let $s\ge1$, put $d=\gcd(n,s)$, and consider the finite \'etale algebra
\[
R=\F_{q^n}\otimes_{\F_q}\F_{q^s}
\cong\prod_{\tau\in\Gamma}\F_{q^{ns/d}},
\qquad
\Gamma=\Gal(\F_{q^d}/\F_q).
\]
Indeed, if $m(Y)$ is the minimal polynomial over $\F_q$ of a generator of $\F_{q^s}$, then $R\cong\F_{q^n}[Y]/(m(Y))$, and $m$ has $d$ irreducible factors of degree $s/d$ over $\F_{q^n}$. We regard $\F_{q^n}$ and $\F_{q^s}$ as subrings of $R$ via $x\mapsto x\otimes1$ and $z\mapsto1\otimes z$. Their images intersect in $\F_q$, also when $d>1$, as one sees by writing elements of $R$ in a basis $\{a_i\otimes a'_j\}$, where $(a_i)$ and $(a'_j)$ are $\F_q$-bases of $\F_{q^n}$ and $\F_{q^s}$ with $a_1=a'_1=1$. The algebra has dimension $s$ over $\F_{q^n}$ and dimension $n$ over $1\otimes\F_{q^s}$.

Fix all fields in an algebraic closure. Identify $\Gamma$ with the additive group $\mathbb Z/d\mathbb Z$ by writing $\tau=\Frob_q^j$ on $\F_{q^d}$, with $0\le j<d$, and let $\widetilde\tau=\Frob_q^j$ on $\F_{q^s}$. We use the component convention
\[
x\otimes z\longmapsto\bigl(x\widetilde\tau(z)\bigr)_{\tau\in\Gamma}.
\]
This is an isomorphism of $\F_{q^n}$-algebras: the $s$ embeddings of $\F_{q^s}$ into $\F_{q^{ns/d}}$ are the restrictions of $\Frob_q^j$, $0\le j<s$, and they form $d$ orbits under $\Gal(\F_{q^{ns/d}}/\F_{q^n})=\langle\Frob_q^n\rangle$, represented by $0\le j<d$.
Let $e_\tau$ be the corresponding primitive idempotents. For $T\subseteq\Gamma$, put $e_T=\sum_{\tau\in T}e_\tau$, and for $y\in R^t$ let
\[
T(y)=\{\tau\in\Gamma:e_\tau y\ne0\}.
\]
Choose $\ell,t$ with $n\ell=(h+1)t$, and set $r=s\ell\ge h+1$, $N=st$. For a uniform matrix $A\in\Mat_{\ell\times t}(R)$, define
\[
U_A=A(1\otimes\F_{q^s})^t\le R^\ell\cong\F_{q^n}^{r}.
\]
If $(b_i)$ is an $\F_q$-basis of $\F_{q^s}$, then $(1\otimes b_i)$ is an $\F_{q^n}$-basis of $R$. Expanding in this basis defines $\Phi_{\mathcal B}:R^t\to\F_{q^n}^{N}$ and $\rk_q(y)=\rk_q(\Phi_{\mathcal B}(y))$ exactly as before.

For an $(h+1)$-dimensional $\F_q$-subspace $D\le\F_{q^n}$, put
\[
\kappa_T(D)=\dim_{\F_{q^s}}\bigl((D\otimes_{\F_q}\F_{q^s})\cap e_T R\bigr),
\]
where dimension is taken using the second scalar embedding.

\begin{theorem}\label{thm:etale}
With the notation above, define
\begin{equation}\label{eq:theta}
\Theta=\frac{q^d-1}{(q^n-1)(q^s-1)}
\sum_{\varnothing\ne T\subsetneq\Gamma}\ \sum_D
\bigl(q^{N\kappa_T(D)}-1\bigr)q^{-N(h+1)|T|/d},
\end{equation}
where $D$ runs over the $(h+1)$-dimensional $\F_q$-subspaces of $\F_{q^n}$. Then
\[
\Prob(U_A\text{ is not maximum }h\text{-scattered})
\le\beta_{n,h,s}(q)(1-q^{-(h+1)N})+\Theta.
\]
For every $T\subseteq\Gamma$,
\[
d\kappa_T(D)\le(h+1)|T|.
\]
When $d=1$, the sum is empty and Theorem~\ref{thm:first} is recovered.
\end{theorem}
\begin{proof}
Multiplication by $x\otimes1$, with $x\in\F_{q^n}^*$, scales all expanded coordinates, while multiplication by $1\otimes z$, with $z\in\F_{q^s}^*$, acts on each block by an element of $\GL(s,q)$. Thus both preserve rank. Expanding the vectors $(1\otimes b_i)v_j$, where $v_j$ is the $j$-th column of $A$, into an $r\times N$ matrix $G$ identifies $Ay=0$ with $G\Phi_{\mathcal B}(y)^{\mathsf T}=0$. Proposition~\ref{prop:criterion} therefore applies without assuming that $R$ is a field.

In the product decomposition of $R$, the components of $A$ are independent uniform matrices. Hence, for $y\ne0$,
\[
\Prob(Ay=0)=q^{-ns\ell|T(y)|/d}=q^{-N(h+1)|T(y)|/d}.
\]
The event and this probability are constant on each $R^*$-orbit. The subgroup
\[
H=(\F_{q^n}^*\otimes1)(1\otimes\F_{q^s}^*)
\]
has order $(q^n-1)(q^s-1)/(q-1)$ and preserves rank and support. On any nonzero component, the condition that $x\otimes z$ acts as $1$ forces $x,z$ to arise from an element of $\F_{q^d}^*$. Thus every stabilizer in $H$ has order at most $(q^d-1)/(q-1)$. On full-support vectors it is trivial. Every $R^*$-orbit meeting the low-rank vectors therefore contains at least $|H|(q-1)/(q^d-1)$ such vectors, and at least $|H|$ when the support is full.

We apply the union bound to the $R^*$-orbits of nonzero vectors of rank at most $h+1$. Such vectors lie in $(D\otimes\F_{q^s})^t$ for some $D$. For a fixed proper support $T$, the vectors of $(D\otimes\F_{q^s})^t$ with support contained in $T$ form an $\F_{q^s}$-space of dimension $t\kappa_T(D)$, so there are at most $\sum_D(q^{N\kappa_T(D)}-1)$ vectors of rank at most $h+1$ with support $T$. Dividing by the lower bound on the number of such vectors in an orbit, and multiplying by $q^{-N(h+1)|T|/d}$, gives the term of~\eqref{eq:theta} indexed by $T$. For full support, the same argument as in Theorem~\ref{thm:first} gives the term $\beta_{n,h,s}(q)(1-q^{-(h+1)N})$.

It remains to bound $\kappa_T(D)$. Choose an $\F_q$-basis $x_1,\ldots,x_{h+1}$ of $D$ and put
\[
K_D=\left\{c\in\F_{q^d}^{h+1}:\sum_{i=1}^{h+1}c_i x_i=0\right\}.
\]
Choose an $\F_{q^d}$-basis $u_1,\ldots,u_{s/d}$ of $\F_{q^s}$ and write $z_i=\sum_j c_{ij}u_j$, with $c_{ij}\in\F_{q^d}$. The $\tau$-component of $\sum_i x_i\otimes z_i$ is
\[
\sum_{j=1}^{s/d}\widetilde\tau(u_j)
\left(\sum_{i=1}^{h+1}x_i\tau(c_{ij})\right).
\]
The fields $\F_{q^n}$ and $\F_{q^s}$ are linearly disjoint over $\F_{q^d}$, so the $\widetilde\tau(u_j)$ are independent over $\F_{q^n}$. This component vanishes precisely when every column $(c_{ij})_i$ lies in $\tau^{-1}K_D$. Requiring vanishing outside $T$ therefore gives $s/d$ independent choices in $X_T=\bigcap_{\tau\notin T}\tau^{-1}K_D$. Dividing the resulting $\F_{q^d}$-dimension by $s/d$ gives
\[
\kappa_T(D)=\dim_{\F_{q^d}}X_T,
\qquad
X_T=\bigcap_{\tau\notin T}\tau^{-1}K_D.
\]
The function $f(T)=\dim_{\F_{q^d}}X_T$ is supermodular, since
$X_T\cap X_{T'}=X_{T\cap T'}$ and $X_T+X_{T'}\subseteq X_{T\cup T'}$. It is invariant under translation in $\Gamma$, because coordinatewise Galois automorphisms permute the spaces $\tau^{-1}K_D$. Also $f(\Gamma)=h+1$ and $f(\varnothing)=0$. For the last assertion, $X_\varnothing$ is Galois-stable and hence has a basis over $\F_q$, whereas $K_D$ contains no nonzero vector of $\F_q^{h+1}$. Here we use the following elementary form of Galois descent: if a subspace of $\F_{q^d}^{h+1}$ is stable under a group $\Delta$ of automorphisms of $\F_{q^d}$ acting coordinatewise, then its reduced row-echelon basis is fixed by $\Delta$, so the subspace has a basis with entries in the fixed field of $\Delta$.

Consider the multiset of the $d$ translates of $T$. Replacing two incomparable members by their union and intersection does not decrease the sum of $f$, preserves the multiplicity of every group element, and strictly increases the sum of the squared cardinalities. The process terminates in a chain. Since every element has multiplicity $|T|$, this chain consists of $|T|$ copies of $\Gamma$ and $d-|T|$ empty sets. Therefore $d f(T)\le(h+1)|T|$.
\end{proof}

The next result determines when equality holds in the inequality $d\kappa_T(D)\le(h+1)|T|$ of Theorem~\ref{thm:etale}.

\begin{theorem}\label{thm:etale-equality}
Let $\varnothing\ne T\subsetneq\Gamma$, let $\Gamma_T$ be its stabilizer under translation, and put $d_T=[\Gamma:\Gamma_T]$. Then
\[
d\kappa_T(D)=(h+1)|T|
\quad\Longleftrightarrow\quad
D\text{ is linear over }\F_{q^{d_T}}.
\]
In particular, equality forces $d_T\mid h+1$. If $D$ is not linear over $\F_{q^{d_T}}$, then
\[
d\kappa_T(D)\le(h+1)|T|-\gcd(d,h+1).
\]
\end{theorem}
\begin{proof}
Write $Y_\tau=\tau^{-1}K_D$ and, for $S\subseteq\Gamma$, let
\[
\varrho(S)=\codim_{\F_{q^d}}\bigcap_{\tau\in S}Y_\tau.
\]
This is a nondecreasing, translation-invariant submodular function, with $\varrho(\Gamma)=h+1$. Theorem~\ref{thm:etale} gives
\[
\varrho(S)\ge\frac{h+1}{d}|S|.
\]
Call $S$ tight when equality holds. Submodularity and the displayed lower bound imply that unions and intersections of tight sets are tight. Translates of tight sets are also tight.

Suppose $S=\Gamma\setminus T$ is tight, and let $\Gamma_T=\Stab(S)$. For each $x\in\Gamma$, the intersection of all translates of $S$ containing $x$ is exactly $x+\Gamma_T$. Indeed, membership of $y$ in this intersection means $x-S\subseteq y-S$; equal cardinalities then give $S+(y-x)=S$. Thus every coset of $\Gamma_T$ is tight, and
\[
\varrho(\Gamma_T)=\frac{h+1}{d_T}.
\]
The space $\bigcap_{\tau\in \Gamma_T}Y_\tau$ is stable under $\Gamma_T$. The vectors fixed by $\Gamma_T$ are precisely the tuples $(c_i)\in\F_{q^{d_T}}^{h+1}$ satisfying $\sum_i c_i x_i=0$. The reduced row-echelon descent argument used in the proof of Theorem~\ref{thm:etale} identifies their dimension over $\F_{q^{d_T}}$ with the dimension of the intersection over $\F_{q^d}$. Consequently,
\[
\varrho(\Gamma_T)=\dim_{\F_{q^{d_T}}}(\F_{q^{d_T}}D).
\]
The preceding two equations imply that $\F_{q^{d_T}}D$ and $D$ have the same $\F_q$-dimension, so they agree.

Conversely, if $D$ is $\F_{q^{d_T}}$-linear, descent gives $\varrho(\Gamma_T)=(h+1)/d_T$. The same holds on each coset by translation. Since $S$ is a union of cosets, subadditivity gives $\varrho(S)\le(h+1)|S|/d$, and the reverse inequality was already proved. This establishes equality. In the strict case, the positive integer $(h+1)|T|-d\kappa_T(D)$ is divisible by $\gcd(d,h+1)$, proving the final assertion.
\end{proof}

\subsection{A uniform bound for the auxiliary degree}

We now use Theorem~\ref{thm:etale-equality} to bound $\Theta$ when $s>h(n-h-2)$.

\begin{lemma}\label{lem:uniform-etale}
Let $2\le h\le n-4$ and $g=\gcd(n,h+1)>1$. Suppose
\[
s>h(n-h-2),\qquad g\nmid s.
\]
For every prime power $q$, the construction $U_A$ of Theorem~\ref{thm:etale} with
\[
\ell=\frac{h+1}{g},\qquad t=\frac ng
\]
gives a maximum $h$-scattered subspace of $V(s(h+1)/g,q^n)$.
\end{lemma}
\begin{proof}
Put $d=\gcd(n,s)$ and $\gamma=\gcd(d,h+1)$. When $d=1$, Corollary~\ref{cor:blocks} applies. Hence assume $d>1$. Since $g\nmid s$, we have $d<n$ and therefore $d\le n/2$. Moreover,
\[
h(n-h-2)\ge2(n-4),\qquad s\ge2n-7\ge5.
\]
Here $N=sn/g$, and $r=s(h+1)/g\ge h+1$ because $s>h(n-h-2)\ge2h\ge g$. Since $d\mid n$, we have $\gamma=\gcd(d,g)$. Since $\lcm(d,g)=dg/\gamma$ divides $n$, it follows that
\[
\frac{N\gamma}{d}\ge s.
\]

For each proper nonempty support $T$, split the subspaces $D$ in~\eqref{eq:theta} according to Theorem~\ref{thm:etale-equality}. If $d\kappa_T(D)<(h+1)|T|$, the term of $D$ is less than $q^{-N((h+1)|T|-d\kappa_T(D))/d}\le q^{-N\gamma/d}$. If $d\kappa_T(D)=(h+1)|T|$, the term of $D$ is less than $1$, and $D$ is linear over $\F_{q^{d_T}}$, where $d_T>1$ divides $\gamma$. There are
\[
\gaussf{n/d_T}{(h+1)/d_T}{q^{d_T}}
\]
such subspaces. For each divisor $d'$ of $d$, at most $2^{d'}-2$ proper nonempty subsets have stabilizer of index $d'$, because they are unions of cosets of the unique subgroup of index $d'$. Consequently,
\begin{equation}\label{eq:theta-split}
\begin{split}
\Theta\le\frac{q^d-1}{(q^n-1)(q^s-1)}\biggl(&
(2^d-2)\gauss{n}{h+1}q^{-N\gamma/d}\\
&+\sum_{\substack{d'\mid\gamma\\d'>1}}(2^{d'}-2)
\gaussf{n/d'}{(h+1)/d'}{q^{d'}}\biggr).
\end{split}
\end{equation}

Write $\omega_k(Q)=\prod_{i=k}^{\infty}(1-Q^{-i})$. The Gaussian product formula gives
\[
\beta_{n,h,s}(q)
\le\frac{q^{h(n-h-2)-s}}{(1-q^{-s})\omega_2(q)}.
\]
Using $N\gamma/d\ge s$, $q^d-1<q^d$ and $\gaussf{M}{K}{Q}\le Q^{K(M-K)}/\omega_1(Q)$ in~\eqref{eq:theta-split}, we see that the failure probability in Theorem~\ref{thm:etale} is strictly less than
\begin{equation}\label{eq:psi}
\begin{split}
\Psi_s(q)={}&
\frac{q^{h(n-h-2)-s}}{(1-q^{-s})\omega_2(q)}
\left(1+\frac{(2^d-2)(q^d-1)}{(q-1)q^s}\right)\\
&+\sum_{\substack{d'\mid\gamma\\d'>1}}
\frac{(2^{d'}-2)q^{d-n-s+(h+1)(n-h-1)/d'}}
{(1-q^{-n})(1-q^{-s})\omega_1(q^{d'})}.
\end{split}
\end{equation}
For fixed $d$, this expression decreases with $s$ and with $q$. Indeed, after expanding $(q^d-1)/(q-1)$, every power of $q$ in the numerators has negative exponent. For the terms in the sum, use $d'\ge2$ and
\[
(h+1)(n-h-1)=h(n-h-2)+n-1
\]
to bound the exponent by $d-(h(n-h-2)+n+3)/2<0$. It therefore suffices to show $\Psi_s(2)<1$.

We use only the elementary estimates
\[
\omega_2(2)>\frac9{16},\qquad
\omega_1(4)>\frac23,\qquad
\omega_1(8)>\frac67.
\]
The first follows by separating the factor $1-2^{-2}$ and using
$\prod_i(1-u_i)\ge1-\sum_i u_i$ on the tail. The other two follow from the same inequality applied to the whole product.

Suppose first that $n\ge11$. Then $s\ge15$, and the first factor in~\eqref{eq:psi} is less than
\[
\frac{1/2}{(1-2^{-15})(9/16)}<\frac{89}{100}.
\]
The extra term in its parentheses is less than $2^{2d-s}\le2^{7-n}$. For $d'=2$, the exponent in the second line of~\eqref{eq:psi} is at most $-n+5/2$, so its contribution before the two factors $1-2^{-n}$ and $1-2^{-s}$ is less than
\[
3\cdot2^{5/2}\,2^{-n}<17\cdot2^{-n}.
\]
For $d'\ge3$, the exponent is at most $d-2n+4$. Since
$\sum_{j=3}^d 2^j<2^{d+1}$, these terms together contribute less than
\[
\frac76\,2^{2d-2n+5}\le\frac{112}{3}\,2^{-n}<38\cdot2^{-n}.
\]
Also $(1-2^{-n})(1-2^{-s})>999/1000$. Combining the estimates yields
\[
\Psi_s(2)<\frac{89}{100}\bigl(1+2^{7-n}\bigr)+\frac{1000}{999}(17+38)\,2^{-n}
<\frac{89}{100}+170\cdot2^{-n}
\le\frac{89}{100}+\frac{170}{2^{11}}<1.
\]

For $n\le10$, the pairs with $g>1$ are
\[
(6,2),\ (8,3),\ (9,2),\ (9,5),\ (10,3),\ (10,4),\ (10,5).
\]
When $n=9$, a noncoprime degree has $g\mid s$ and is excluded. All other possibilities are in the following table. Here $s_0$ is the least integer $s>h(n-h-2)$ with $g\nmid s$ and $\gcd(n,s)=d$. Substitution in~\eqref{eq:psi}, using $\omega_2(2)>9/16$ and $\omega_1(2^{d'})>(2^{d'}-2)/(2^{d'}-1)$, gives the displayed rational bounds.
\begin{center}
\begin{tabular}{@{}rrrrl@{}}
\toprule
$n$&$h$&$d$&$s_0$&upper bound for $\Psi_{s_0}(2)$\\
\midrule
6&2&2&8&$262/2295<1$\\
8&3&2&10&$709616/782595<1$\\
10&4&2&18&$149800/337041<1$\\
10&3,5&5&25&$16777681/9663676128<1$\\
\bottomrule
\end{tabular}
\end{center}
Monotonicity in $s$ finishes the proof.
\end{proof}

\subsection{Existence in every sufficiently large admissible dimension}

\begin{proof}[Proof of Theorem~\ref{thm:main}]
If $g=1$, every admissible $r$ is divisible by $h+1$, so the standard construction suffices. The cases $h=1$ and $h=n-3$ follow from the known constructions cited in the introduction. The case $h=n-2$ has $g=1$.

It remains to take $2\le h\le n-4$ and $g>1$. Write
\[
r=\frac{h+1}{g}b.
\]
Here $b$ is an integer, because $(h+1)/g$ divides $r(n/g)$ and is coprime to $n/g$. We call $b=gr/(h+1)$ the \emph{normalized dimension} of $V(r,q^n)$. If $g\mid b$, use standard blocks. Otherwise, Lemma~\ref{lem:uniform-etale}, with $s=b$, gives the required subspace as soon as $b>h(n-h-2)$.
\end{proof}

\begin{remark}\label{rem:uniform-blocks}
Assume $2\le h\le n-4$ and $g>1$. The proof can also be expressed through finitely many blocks. For each $j=1,\ldots,g-1$, choose the unique integer $\sigma_j$ in
\[
\{h(n-h-2)+1,\ldots,h(n-h-2)+g\}
\]
congruent to $j$ modulo $g$. Standard blocks have normalized dimension $g$, and Lemma~\ref{lem:uniform-etale} supplies the normalized dimensions $\sigma_j$. Their direct sums give every normalized dimension at least $h(n-h-2)+1$. This is the exact threshold for these generators: $h(n-h-2)\equiv1\pmod g$, so the preceding integer is not a multiple of $g$ and is smaller than every $\sigma_j$.
\end{remark}

\begin{corollary}\label{cor:mrd}
Let $n\ge3$, $2\le\delta\le n-1$ and $g=\gcd(n,\delta-1)$. For every prime power $q$, there exists an $\F_{q^n}$-linear MRD code with parameters
\[
\left[N,\frac{(n-\delta+1)N}{n},\delta\right]_{q^n/q}
\]
whenever
\[
\frac ng\mid N,\qquad
N\ge\frac ng\bigl((n-\delta)(\delta-2)+1\bigr).
\]
\end{corollary}
\begin{proof}
Apply Theorem~\ref{thm:main} with $h=n-\delta$ and $r=(n-\delta+1)N/n$, and then use Proposition~\ref{prop:criterion}. The identity $\gcd(n,n-\delta+1)=\gcd(n,\delta-1)$ gives the stated divisibility condition.
\end{proof}

We compare Theorem~\ref{thm:main} with the thresholds obtained from Corollary~\ref{cor:blocks} alone. For each unit residue $j$ modulo $g$, let
\[
s_j=\min\{s>h(n-h-2):\gcd(s,n)=1,\ s\equiv j\pmod g\}.
\]
The normalized dimensions supplied by Corollary~\ref{cor:blocks} and standard blocks under direct sums form the numerical semigroup
\[
\mathcal S=\langle g,\ s_j:j\in(\mathbb Z/g\mathbb Z)^*\rangle.
\]
Indeed, every integer $s>h(n-h-2)$ coprime to $n$ is some $s_j$ plus a nonnegative multiple of $g$, and the blocks of Corollary~\ref{cor:blocks} with auxiliary degree $s$ have normalized dimension a multiple of $s$. For $0\le j<g$, let $M_j$ be the least element of $\mathcal S$ congruent to $j$ modulo $g$. Since $M_j-g$ is the largest integer congruent to $j$ modulo $g$ that does not lie in $\mathcal S$, the conductor of $\mathcal S$ is
\[
c(\mathcal S)=\max_{0\le j<g}M_j-g+1.
\]
Thus the least threshold obtained from Corollary~\ref{cor:blocks} and standard blocks is $\max\{(h+1)c(\mathcal S)/g,\,h+1\}$. Table~\ref{tab:thresholds} compares it with the threshold of Theorem~\ref{thm:main}.

\begin{table}[ht]
\centering
\begin{tabular}{@{}rrrrr@{}}
\toprule
$n$&$h$&$g$&coprime threshold&\'etale threshold\\
\midrule
8&3&4&19&10\\
12&3&4&43&22\\
12&5&6&82&26\\
12&7&4&86&44\\
15&2&3&26&23\\
16&7&8&101&50\\
16&11&4&201&102\\
18&8&9&130&65\\
18&11&6&344&112\\
24&11&12&370&122\\
\bottomrule
\end{tabular}
\caption{Thresholds $r_0$ such that maximum $h$-scattered subspaces of $V(r,q^n)$ are obtained for every $q$ and every $r\ge r_0$ with $(h+1)\mid rn$: from Corollary~\ref{cor:blocks} and standard blocks (coprime threshold, computed from the conductor of $\mathcal S$), and from Theorem~\ref{thm:main} (\'etale threshold).}\label{tab:thresholds}
\end{table}

\begin{example}\label{ex:uniform}
For $(n,h)=(8,3)$, the coprime semigroup is $\langle4,11,13\rangle$, of conductor $19$. The \'etale construction also supplies $s=10$, and $\langle4,10,11,13\rangle$ contains every integer at least $10$. Hence maximum $3$-scattered subspaces exist for every $q$ in dimensions $4,8$ and every $r\ge10$.
\end{example}

\section{Existence without divisibility}\label{sec:rounded}

Put
\[
N=\left\lfloor\frac{rn}{h+1}\right\rfloor,\qquad
\epsilon=rn-(h+1)N,\qquad 0\le\epsilon\le h.
\]
For an $\F_q$-subspace $U$ of $V(r,q^n)$, we call $\epsilon(U)=rn-(h+1)\dim_{\F_q}U$ the \emph{defect} of $U$. If $r\ge h+1$ and $U$ is $h$-scattered, then $\epsilon(U)\ge0$, and $U$ attains the integer upper bound if and only if $\epsilon(U)=\epsilon$, that is, if and only if $\epsilon(U)\le h$. The defect is additive under direct sums. The subspace $\F_q$ of $V(1,q^n)=\F_{q^n}$ has dimension one and defect $n-h-1$.

\begin{proposition}\label{prop:free}
Let $r_0\ge h+1$ and let $U_0$ be an $h$-scattered subspace of $V(r_0,q^n)$ attaining the integer upper bound. Put $N_0=\dim_{\F_q}U_0$ and $\epsilon_0=r_0n-(h+1)N_0$. For every $j\ge0$, the subspace
\[
U_0\oplus\F_q^j\le\F_{q^n}^{r_0+j}
\]
is $h$-scattered and has dimension $N_0+j$. It attains the integer upper bound if and only if
\[
\epsilon_0+j(n-h-1)<h+1.
\]
\end{proposition}
\begin{proof}
If $G_0$ is an $\F_q$-basis matrix of $U_0$, append an identity block of size $j$. The resulting matrix has full row rank and independent columns over $\F_q$, and its kernel is $\ker G_0\oplus\{0\}$. Lemma~\ref{lem:generalcriterion} proves that its column space is $h$-scattered. Finally,
\[
n(r_0+j)-(h+1)(N_0+j)=\epsilon_0+j(n-h-1).
\]
This nonnegative integer is less than $h+1$ exactly when $N_0+j=\lfloor n(r_0+j)/(h+1)\rfloor$.
\end{proof}

\begin{corollary}\label{cor:residue-classes}
Put $g=\gcd(n,h+1)$ and $J=\lfloor h/(n-h-1)\rfloor$. For every prime power $q$, the integer upper bound is attained whenever
\[
r=\frac{h+1}{g}b+j,\qquad
b\ge h(n-h-2)+1,\qquad 0\le j\le J.
\]
Standard blocks also give every $r=u(h+1)+j$ with $u\ge1$ and $0\le j\le J$.
\end{corollary}
\begin{proof}
Since $h(n-h-2)+1\ge g$, the ambient dimension $(h+1)b/g$ is at least $h+1$. Theorem~\ref{thm:main} supplies a maximum block in that dimension, and Proposition~\ref{prop:free} allows exactly the stated values of $j$. The second assertion uses a direct sum of $u$ standard blocks.
\end{proof}

\begin{proof}[Proof of Theorem~\ref{thm:mainrounded}]
Put $g=\gcd(n,h+1)$. Since $n=(h+1)+(n-h-1)$, the hypothesis gives $g=n-h-1$. Write
\[
r=\frac{h+1}{g}b+j,\qquad 0\le j<\frac{h+1}{g}.
\]
The assumed bound on $r$ implies $b\ge h(n-h-2)+1$, and $jg<h+1$. Corollary~\ref{cor:residue-classes} applies.
\end{proof}

The values $j=0,\ldots,J$ are distinct modulo $(h+1)/g$, since $g\mid n-h-1$ implies $J<(h+1)/g$. If $(h+1)/g>1$, these operations cover every eventual residue precisely when $(n-h-1)\mid(h+1)$. For the remaining residues this construction gives no information.

\begin{example}
For $(n,h)=(12,8)$, one has $g=3$ and $h(n-h-2)+1=17$. Hence every $r\ge51$ admits an $8$-scattered subspace of dimension $\lfloor4r/3\rfloor$, for every $q$. At $r=53$, a maximum block with $(r,N)=(51,68)$ and two free coordinates give dimension $70$.
\end{example}

For $h=n-2$, standard blocks and free coordinates attain the integer upper bound for every $r\ge h+1$. The same holds when $h=n-3$ and $n$ is even. Indeed, the known maximum constructions obtained by Delsarte duality \cite[Corollary~3.5]{CMPZ} give every ambient dimension at least $h+1$ divisible by $(h+1)/2$. Choose the largest such multiple not exceeding $r$, then append fewer than $(h+1)/2$ coordinates. Each has defect two.

\subsection{The direct bound at positive defect}

An auxiliary degree is compatible with $(r,N)$ if
\begin{equation}\label{eq:compatible}
s\mid\gcd(r,N),\qquad \gcd(s,n)=1.
\end{equation}
For such an $s$, write $r=s\ell$ and $N=st$. Take a random $\ell\times t$ matrix over $\F_{q^{ns}}$ and consider the image of $\F_{q^s}^t$. The proof of Theorem~\ref{thm:first} gives the following upper bound on the probability that the image fails to be $h$-scattered of dimension $N$:
\begin{equation}\label{eq:bound}
E_s=q^{-\epsilon}\beta_{n,h,s}(q)(1-q^{-(h+1)N}),
\end{equation}
where $\beta_{n,h,s}(q)$ is defined in Theorem~\ref{thm:first}. Indeed, the counting in the proof of Theorem~\ref{thm:first} applies without change, except that a fixed nonzero vector lies in the kernel with probability $q^{-nr}=q^{-(h+1)N-\epsilon}$.

If $E_s<1$, some expanded matrix has no nonzero kernel vector of rank at most $h+1$. Proposition~\ref{prop:criterion} then gives an $h$-scattered subspace of dimension $N$ spanning the prescribed ambient space.

For fixed $n,h,r,q$, expression~\eqref{eq:bound} strictly decreases with $s$. Therefore the best compatible choice is the largest divisor of $\gcd(r,N)$ coprime to $n$. It is obtained by removing from $\gcd(r,N)$ every prime-power factor whose prime divides $n$.

\begin{theorem}\label{thm:direct}
Let $s$ satisfy~\eqref{eq:compatible}. Then $E_s<1$ holds exactly in the following cases:
\begin{enumerate}
\item $\epsilon=0$ and $s>h(n-h-2)$;
\item $\epsilon>0$ and $h=n-2$;
\item $\epsilon>0$, $h=n-3$, $s+\epsilon>h(n-h-2)$, and either $s\ge2$ or $q\ge3$.
\end{enumerate}
In particular, if $h\le n-4$ and $\epsilon>0$, the displayed bound never proves existence at the integer upper bound.
\end{theorem}

\begin{proof}
Put $a=\min\{h+1,n-h-1\}$. By the proof of Lemma~\ref{lem:beta},
\[
q^{-\epsilon}\beta_{n,h,s}(q)
=q^{h(n-h-2)-s-\epsilon}
\frac{\prod_{j=n-a+1}^{n-1}(1-q^{-j})}{(1-q^{-s})\prod_{j=2}^{a}(1-q^{-j})}.
\]
For $a\ge2$, the quotient of the two products over $j$ is larger than $1$, as shown in that proof. Consequently, if $s+\epsilon\le h(n-h-2)$, then
\[
q^{-\epsilon}\beta_{n,h,s}(q)>(1-q^{-s})^{-1}.
\]
Since $(h+1)N>s$, the remaining factor $1-q^{-(h+1)N}$ does not bring $E_s$ below one.

For $\epsilon=0$, the converse follows from Lemma~\ref{lem:beta}, including its equality case because of the strict factor $1-q^{-(h+1)N}$ in~\eqref{eq:bound}. This proves the first case.

Suppose next that $\epsilon>0$. Compatibility implies $s\mid\epsilon$, so $s\le\epsilon\le h$. If $h\le n-4$, then
\[
s+\epsilon\le2h\le h(n-h-2),
\]
and the bound fails. If $h=n-2$, the formula $\beta_{n,h,s}(q)=1/(q^s-1)$ proves the assertion immediately.

For $h=n-3$, one has $h(n-h-2)=h$ and
\[
q^{-\epsilon}\beta_{n,h,s}(q)
=q^{h-s-\epsilon}
\frac{1-q^{-(h+2)}}{(1-q^{-2})(1-q^{-s})}.
\]
If $s+\epsilon>h(n-h-2)$ and $s\ge2$, the right-hand side is less than $8/9$. If $s=1$, the strict inequality forces $\epsilon=h$. For $q\ge3$ the right-hand side is less than $9/16$. For $q=2$, feasibility forces $h$ even and $h\ge2$; the expression equals $4(1-2^{-(h+2)})/3\ge5/4$, while $(h+1)N\ge15$, so $E_s>1$. This proves all cases.
\end{proof}

The bound~\eqref{eq:bound} is not sharp: a vector of rank at most $h$ is counted once for each $D$ containing its coordinates, a point of $\mathcal O$ may contain several $H$-orbits of vectors of expanded rank at most $h+1$, and the union bound is not exact. Theorem~\ref{thm:direct} only determines when this bound is less than one.

\subsection{Additional constructions when \texorpdfstring{$h=n-3$}{h = n-3}}

It remains to classify the positive-defect auxiliary blocks for $n=h+3$. Write $c=r\bmod (h+1)$, with $0\le c< h+1$. If $c<(h+1)/2$, standard blocks and $c$ free coordinates already attain the integer upper bound. If $c>(h+1)/2$, put
\[
\epsilon_c=2c-(h+1)>0.
\]
The middle class $c=(h+1)/2$, when it exists, has defect zero and is covered by the known maximum constructions.

\begin{proposition}\label{prop:m2}
For each $c>(h+1)/2$, consider
\[
\mathcal D_c=\{s\in\mathbb Z:\ s\mid\epsilon_c,\ s\ge3\text{ odd},\ \gcd(s,n)=1,\ s+\epsilon_c\ge h+1\}.
\]
If this set is nonempty, take its smallest element $s_c$ and define
\begin{equation}\label{eq:positiveatoms}
r_c=\frac{(h+1) s_c+\epsilon_c}{2},\qquad
N_c=\frac{n s_c+\epsilon_c}{2}.
\end{equation}
For every $q$, the bound~\eqref{eq:bound} gives an $h$-scattered subspace of dimension $N_c$ in $V(r_c,q^n)$. The parameters of every other positive-defect construction supplied by that bound are obtained from these blocks, standard blocks and free coordinates.
\end{proposition}

\begin{proof}
The oddness condition is automatic: if $h+1$ is odd then $\epsilon_c$ is odd, while if $h+1$ is even then $n$ is even and $s$ is coprime to $n$. Thus~\eqref{eq:positiveatoms} is integral. Dividing by $s=s_c$ gives
\[
\ell=\frac{h+1+\epsilon_c/s}{2},\qquad
t=\frac{n+\epsilon_c/s}{2}=\ell+1.
\]
Theorem~\ref{thm:direct} applies.

Now fix $s\in\mathcal D_c$ and let $r\equiv c\pmod{h+1}$ be compatible with $s$, where $N=\lfloor rn/(h+1)\rfloor$. Put $m=N-r$. Since $n=h+3$, we have $\epsilon=rn-(h+1)N=2r-(h+1)m$, and $r\equiv c\pmod{h+1}$ gives $\epsilon=\epsilon_c$. Hence $2r=m(h+1)+\epsilon_c$, and $s\mid m$ because $s$ divides $r$ and $N$. As $\epsilon_c=2c-(h+1)$, this gives $2(r-c)=(h+1)(m-1)$, and $r\equiv c\pmod{h+1}$ shows that $m$ is odd. Hence $m$ is an odd positive multiple of $s$, and
\[
r=\frac{(h+1)s+\epsilon_c}{2}+(h+1)s\,\frac{m/s-1}{2},
\qquad
\frac{(h+1)s+\epsilon_c}{2}-r_c=(h+1)\frac{s-s_c}{2},
\]
where $(m/s-1)/2$ and $(s-s_c)/2$ are nonnegative integers. Thus $(r,N)$ is obtained from $(r_c,N_c)$ by adding standard blocks.

For $c<(h+1)/2$ and $r=u(h+1)+c\ge h+1$, one has $u\ge1$, and $u$ standard blocks together with $c$ free coordinates have defect $2c<h+1$. Thus these lower classes need no auxiliary block. The only $q$-dependent direct constructions in Theorem~\ref{thm:direct} have $s=1$, $\epsilon=h$, and even $h$. They belong to $c=h/2$, a lower residue already covered for every $q$.
\end{proof}

\begin{example}
For $(n,h)=(7,4)$, take $c=4$, $\epsilon_c=3$, and $s_c=3$. This gives $(r_c,N_c)=(9,12)$ and therefore
\[
(r,N)=(9+5k,12+7k),\qquad k\ge0.
\]
For $(n,h)=(13,10)$, the upper classes are $c=6,\ldots,10$, and $\mathcal D_c\ne\varnothing$ only for $c=9,10$; the corresponding blocks $(r,N)$ are $(42,49)$ and $(21,24)$. In the second case $\epsilon_c=9$ but $s_c=3$, so taking $s=\epsilon_c$ would give a worse threshold.
\end{example}

\begin{corollary}\label{cor:oddm2}
Consider maximum blocks, the positive-defect blocks above, free coordinates and their direct sums. If $n=h+3$ is odd, these constructions attain the integer upper bound for all sufficiently large $r$ in a residue class $c$ modulo $h+1$ if and only if $c$ is one of
\[
0,1,\ldots,\frac{h}{2}
\quad\text{and}\quad
\left\lceil\frac{3(h+1)}{4}\right\rceil,\ldots,h.
\]
For $h=2$ the second interval is empty. If $n$ is even, they attain it for every $r\ge h+1$.
\end{corollary}

\begin{proof}
Suppose first that $n$ is odd, so $h$ is even. Upper residues have odd defect. By $s\le\epsilon$ and $s+\epsilon>h$, every such auxiliary block has $\epsilon>h/2$. Let $\epsilon_*$ be the smallest odd integer satisfying this inequality. For $h\ge4$, it is at most $h-1$ and is coprime to $n$: according to $n\bmod4$, it equals $(n-1)/2$ or $(n+1)/2$. Thus $\epsilon_*\in\mathcal D_c$ for $c=(h+1+\epsilon_*)/2$. Each free coordinate increases the defect by two, so appending free coordinates produces every odd defect from $\epsilon_*$ through $h-1$. No sum can produce a smaller odd defect, since all free defects are even and every odd auxiliary defect is at least $\epsilon_*$. Translating back by $c=(h+1+\epsilon)/2$ gives the assertion. The case $h=2$ is immediate. If $n$ is even, then $h+1=n-2$ is even, so the known maximum constructions and free coordinates cover every ambient dimension at least $h+1$.
\end{proof}

\section*{Acknowledgments}
This research has been partially supported by the Italian National Group for Algebraic and Geometric Structures and their Applications (GNSAGA -- INdAM). A. Neri is supported by the University of Naples Federico II through the FRA2024 Project CUP E63C26001530001  ``NHACT -- New Horizons in Algebraic Coding Theory''.

\section*{Declaration on the use of AI tools}
The authors used Claude (Anthropic, Opus 5.5) to improve the exposition, check mathematical arguments, and assist in the development of proofs based on ideas provided by the authors. All results and proofs were independently verified by the authors, who assume full responsibility for the content of the manuscript.

\bibliographystyle{abbrv}
\bibliography{references}
\end{document}